\documentclass[11pt]{amsart}
\usepackage[active]{srcltx}
\usepackage{amsmath,amssymb}
\usepackage{latexsym}
\usepackage{color}
\usepackage{dsfont}
\usepackage[normalem]{ulem}
\usepackage{epsfig}
\usepackage{pgf,pgfarrows,pgfnodes,pgfautomata,pgfheaps}
\usepackage{multicol}
\usepackage{multirow}
\newtheorem{teo}{Theorem}[section]
\newtheorem{lema}[teo]{Lemma}

\newtheorem{cor}[teo]{Corollary}
\newtheorem{rem}{Remark}[section]

\usepackage{epstopdf}

\begin{document}

\title[A localized reaction-diffusion equation]{A nonlinear diffusion equation with reaction localized to the half-line}

\author{R. Ferreira and A. de Pablo}

\address{Ra\'{u}l Ferreira
\hfill\break\indent  Departamento de Matem\'{a}ticas,
\hfill\break\indent U.~Complutense de Madrid,
\hfill\break\indent
 28040 Madrid, Spain.
\hfill\break\indent  e-mail: {\tt raul$_-$ferreira@mat.ucm.es}}

\address{Arturo de Pablo
\hfill\break\indent  Departamento de Matem\'{a}ticas,
\hfill\break\indent U. Carlos III de Madrid,
\hfill\break\indent
 28911 Legan\'{e}s, Spain.
\hfill\break\indent  e-mail: {\tt arturop@math.uc3m.es}}

\maketitle

\begin{abstract}
We study the behaviour of  the solutions to the quasilinear heat equation  with a reaction restricted to a half-line
$$
u_t=(u^m)_{xx}+a(x) u^p,
$$
$m, p>0$ and $a(x)=1$ for $x>0$, $a(x)=0$ for $x<0$. We first characterize the global existence exponent $p_0=1$ and the Fujita exponent $p_c=m+2$. Then we pass to study the grow-up rate in the case $p\le1$ and the blow-up rate for $p>1$. In particular we show that the grow-up rate is different as for global reaction if $p>m$ or $p=1\neq m$.
\end{abstract}

\

\emph{In memoriam of our friend Ireneo Peral. Master of Mathematics}

\

\section{Introduction}

We consider the following Cauchy problem
\begin{equation}\label{1.1}
\left\{
\begin{array}{ll}
u_t=(u^m)_{xx}+a(x)u^p, \quad & x\in\mathbb R,\; t>0,\\
u(x,0)=u_0(x).
\end{array}\right.
\end{equation}
We take exponents $m,p>0$ and the coefficient is the characteristic function of a half-line,
$a(x)=\mathds{1}_{(0,\infty)}(x)$. The initial value $u_0\in L^1(\mathbb{R})\cap L^\infty(\mathbb{R})$ is assumed to be continuous and nonnegative, so that nonnegative solutions $u\ge0$ are considered. We are interested in characterizing and describing the phenomena of blow-up and grow-up for the solutions to \eqref{1.1} in terms of the parameters of the problem, the exponents $m$ and $p$ and the initial datum $u_0$. By a solution $u$ having blow-up we mean that there exists a finite time $T$ such that $u$ is well defined and finite for $t<T$ and
\begin{equation}
\lim_{t\to T^-}\|u(\cdot,t)\|_\infty=\infty.
\end{equation}
When $T=\infty$ we say that $u$ has grow-up.

The problem with global reaction  $a(x)=1$ has been deeply studied in the last years mainly concerning blow-up and $p>1$, see for instance the survey book \cite{SamarskiiBook87}, but also in relation to grow-up, and thus $p\le1$, see \cite{AE86,dPV91}. In fact  there can exist blow-up solutions only if $p>1$, and in that range small initial data produce global solutions if and only if $p>m+2$. The global solutions are unbounded if $p\le1$, i.e. they have grow-up, while they are globally bounded if $p>m+2$. The exponents $p_0=1$ and $p_c=m+2$ are called, respectively, global existence exponent and Fujita exponent.
For the related case in which the reaction coefficient is $a(x)=\mathds{1}_{(-L,L)}(x)$, $0<L<\infty$, the exponents are  $p_0=\max\{1,\frac{m+1}2\}$ and $p_c=m+1$, see \cite{BaiZhouZheng,FdPV06,Pinsky}.

The first result in the paper establishes precisely for which exponents and data we have such phenomena of blow-up or grow-up. We prove that the exponents are the same as for the case $a(x)=1$.
\begin{teo}\label{teo-exponents}\

\begin{enumerate}
\item If $0<p\le p_0=1$ all the solutions to problem \eqref{1.1} are globally defined and unbounded.
\item If $1<p\le p_c=m+2$ all the solutions blow up in finite time.
\item If $p>m+2$  solutions may blow up in finite time or not depending on the initial data. Global solutions are bounded.
\end{enumerate}
\label{main-teo}\end{teo}

The second question to deal with is the speed at which the unbounded  solutions tend to infinity, both in the grow-up and in the blow-up cases. For global solutions we show that these rates are not the natural ones given by the corresponding no diffusion ODE \eqref{flat}. This in fact gives an upper estimate of the grow-up rate by comparison,
\begin{equation}\label{guprate-a=1}
u(x,t)\le  \left\{
\begin{array}{ll}
ct^{\frac1{1-p}}, \qquad &p<1,\\
ce^t,&p=1.
\end{array}\right.
\end{equation}

We remark that when  $p<1$ the reaction function is not Lipschitz, and uniqueness does not necessarily hold, neither comparison, see \cite{AE86,dPV91}. In that case  we can use for comparison a maximal solution or a minimal solution, \cite{dPV91}.

In the case of global reaction $a(x)=1$, it is proved in \cite{AE86,dPV90} that the above is indeed the grow-up rate when $0<p<1$, that is
\begin{equation}
  \label{a=1-p<1}
u(x,t)\sim t^{\frac1{1-p}}
\end{equation}
uniformly in compact sets. By $f\sim g$ we mean $0<c_1\le f/g\le c_2<\infty$.

However, for $p=1$ it is well known, through and easy change of variables that eliminates the reaction, that $u(x,t)\sim
t^{-1/2}e^{t}$ if $m=1$ and
\begin{equation*}\label{decay-a=1}
u(x,t)\sim
e^{\gamma t},\qquad \gamma=\min\{1,\frac{2}{m+1}\},
\end{equation*}
when $m\neq1$, for $t$ large uniformly in compact sets of $\mathbb{R}$, see \cite{vazquez} .

On the other hand, when $a(x)=\mathds{1}_{(-L,L)}$ it is proved in \cite{FdP} that  estimate \eqref{guprate-a=1} is far from being sharp in most of the cases. In particular
$$
u(x,t)\sim
\left\{
\begin{array}{ll}
t^{\frac{1}{m+1-2p}},\qquad & \text{if } p\le 1<m,\\
t^{\frac{1}{1-p}},\qquad & \text{if } m\le p<1,\\
e^t,&  \text{if } m<p=1,
\end{array}\right.
$$
uniformly in compact sets in the first case, only for $|x|<L$ in the last two cases. For $|x|>L$ the rate is different in the case $p>m$, namely
$$
u(x,t)\sim t^{\frac{1}{1-m}}.
$$
In the limit case of linear diffusion and linear reaction, $m=p=1$, it holds
$$
\lim_{t\to\infty}\frac{\log u(x,t)}t=\omega,
$$
where $\omega=\omega(L)\in(0,1)$, $\lim_{L\to\infty}\omega(L)=1$.

For our problem \eqref{1.1} we show that the rate is the same as for global reaction only if $p\le m$ with $p<1$  or $p=m=1$; it is the same as for $a(x)=\mathds{1}_{(-L,L)}$ if $p>m$, and strictly in between of those two problems if $p=1<m$. Again the rate is different for $p>m$ inside or outside the support of the reaction coefficient $a(x)$.

\begin{teo}\label{teo-guprates-p=1} Let  $u$ be a solution to problem \eqref{1.1} with $p=1$.

\begin{enumerate}
\item If $m>1$ then $u(x,t)\sim e^{\alpha t}$ uniformly in compact sets of $\mathbb{R}$, where $\alpha\in(1/m,2/(m+1))$ depends on the behaviour of $u_0$ at infinity.
\item If $m=1$ then $u(x,t)\sim e^t$ uniformly in compact sets of $\mathbb{R}$.
\item If $m<1$ then $u(x,t)\sim e^t$ uniformly in compact sets of $\mathbb{R}^+$ and $u(x,t)\sim t^{\frac1{1-m}}$ uniformly in compact sets of $\mathbb{R}^-$, provided $u_0(x)\sim |x|^{\frac{-2}{1-m}}\left( \log|x|\right)^{\frac{1}{1-m}}$ for $x\sim -\infty$.
\end{enumerate}

\end{teo}

\begin{teo}\label{teo-guprates-p<1} Let  $u$ be a solution to problem \eqref{1.1} with $p<1$.

\begin{enumerate}
\item If $m\ge p$ then $u(x,t)\sim t^{\frac1{1-p}}$ uniformly in compact sets of $\mathbb{R}$.
\item If $m<p$ then $u(x,t)\sim t^{\frac1{1-p}}$ uniformly in compact sets of $\mathbb{R}^+$ and $u(x,t)\sim t^{\frac1{1-m}}$ uniformly in compact sets of $\mathbb{R}^-$, provided $u_0(x)\sim |x|^{\frac{-2}{1-m}}$ for $x\sim -\infty$.
\end{enumerate}
\end{teo}

\

We compare in the next table the grow-up rates for the three problems, with global recation $a(x)=1$, localized reaction $a(x)=\mathds{1}_{(-L,L)}$, and reaction confined to the half-line $a(x)=\mathds{1}_{(0,\infty)}$.

\

\begin{center}

\begin{tabular}{|c|c|c|c|c|c|c|}
\hline
\multirow{2}{1cm}{}&\multicolumn{3}{c|}{$\mathbf{p=1}$}&\multicolumn{3}{c|}{$\mathbf{p<1}$} \\
\cline{2-7}
 &$m>1$ &$m=1$ &$m<1$&$m>p$&$m=p$&$m<p$ \\
\hline
$\mathbb{R}$  &$e^{ct}$ &$e^t$ &$e^t$&$t^a$ &$t^a$ &$t^a$ \\
\hline
$(-L,L)$ &$t^a$&$e^{\omega t}$ &$e^t\,/\,t^b$&$t^d$&$t^d$&$t^a\,/\,t^b$\\
\hline
$(0,\infty)$  &$e^{\alpha t}$& $e^t$& $e^t\,/\,t^b$&$t^a$ &$t^a$&$t^a\,/\,t^b$\\
\hline
\end{tabular}
\end{center}

\

The exponents are
$$
\begin{array}{l}
a=\dfrac1{1-p},\quad b=\dfrac1{1-m},\quad c=\dfrac2{m+1},\quad d=\dfrac1{m+1-2p},  \\
\omega<1 \;\text{ depends on } L,  \\
\alpha\le\alpha^*(m)<c \; \text{ depends on the behaviour of $u_0$ at infinity}.
\end{array}
$$
In the case $p>m$ we have two different rates, inside or outside the support of $a(x)$.

\

As for blow-up, the rate at which the solutions approach infinity in a finite time  has been studied for the case of global reaction under different conditions on the initial datum and exponents, with special care in the multidimensional case, see \cite{SamarskiiBook87} and the references therein. For dimension one, as is our situation, any solution with blow-up at time $t=T$ satisfies, for $t$ close to $T$,
$$
  \|u(\cdot,t)\|_\infty\sim (T-t)^{-\frac1{p-1}}.
  $$

For localized reaction $a(x)=\mathds{1}_{(-L,L)}$ the rates have been established in \cite{BaiZhouZheng,FdPV06}, giving a different rate depending on $p$ being bigger or smaller than $m$,
$$
\|u(\cdot,t)\|_\infty\sim (T-t)^{-\gamma},\qquad \gamma=\max\{\frac1{p-1},\frac1{2p-m-1}\}.
$$
In addition the property $\partial_tu\ge0$ is required in the proof of this result.

We prove here for problem \eqref{1.1} that the rate is the same as for global reaction, assuming again $\partial_tu\ge0$, but this is required only above the Fujita exponent, i.e. for $p>m+2$.

\begin{teo}\label{teo-buprates}
   Let $u$ be a solution to problem \eqref{1.1} with $p>1$ such that becomes infinity for $t\to T^-$, and assume further that $\partial_tu\ge0$ if $p>m+2$. Then
\begin{equation}\label{buprate}
  \|u(\cdot,t)\|_\infty\sim (T-t)^{-\frac1{p-1}}.
\end{equation}
\end{teo}

\

We end the description of solutions of problem \eqref{1.1} by studying the set where the solution tends to infinity, the blow-up set
$$
B(u)=\{x\in\mathbb{R}\,:\,\exists\, x_j\to x,\;t_j\to T,\; u(x_j,t_j)\to\infty\}.
$$
In the global reaction case it has been proved the three possibilities according to the reaction exponent: single point blow-up, $B(u)$ is a discrete set, if $p>m$; regional blow-up, $B(u)$ is a compact set of positive measure, if $p=m$; and  global blow-up, $B(u)=\mathbb{R}$, if $p<m$. See again \cite{SamarskiiBook87}. The same happens for localized reaction $a(x)=\mathds{1}_{(-L,L)}$, at least for $m>1$ and symmetric nondecreasing initial values, see \cite{FdPV06}. In our case we prove that the same happens, and we additionally  show where this blow-up set can lie in the case where the blow-up is not the whole line. To do that we assume in the case $p\ge m$  that there exists some point $x_0$ for which the blow-up rate \eqref{buprate} holds, i.e. 
\begin{equation}\label{x0}
u(x_0,t)\ge c (T-t)^{-\frac1{p-1}}.
\end{equation}
\begin{teo}\label{teo-bupsets}
   Let $u$ be a blow-up solution to problem \eqref{1.1}, with compactly supported initial datum. Assume also \eqref{x0}. We have for the blow-up set $B(u)$:
\begin{enumerate}
  \item if $p>m$ then $B(u)\subset\mathbb{R}^+$. Moreover if $m>1$ it is bounded;
  \item if $p=m$ then $B(u)$ is bounded with nontrivial measure;
  \item if $p<m$ then $B(u)=\mathbb{R}$.
\end{enumerate}
\end{teo}

We remark that due to the lack of symmetry in the problem it is not clear the existence of the point assumed in the statement. In general we can prove that $B(u)=[x_1,\infty)$ for some $-\infty\le x_1<\infty$ if $p<m$, and $B(u)$ is bounded if $p=m$.

\

\noindent\textsc{Organization of the paper}: We characterize the critical exponents,  Theorem~\ref{teo-exponents}, in Sections~\ref{sect-globalexp} and~\ref{sect-fujitaexp}. The grow-up rates, Theorems~\ref{teo-guprates-p=1} and~\ref{teo-guprates-p<1} are proved in Section~\ref{sect-guprates}, while the   blow-up rates, Theorem~\ref{teo-buprates} is proved in Section~\ref{sect-buprates}. Finally we devote Section~\ref{sect-bupsets} to describe the blow-up sets, Theorem~\ref{teo-bupsets}.
%

\section{Blow-up versus global existence}\label{sect-globalexp}

We prove in this section that the global existence exponent is $p_0=1$. First it is obvious that if $0<p\le1$ every solution to problem \eqref{1.1} is global. Just use  comparison with the flat supersolution
\begin{equation}\label{flat}
U'=U^p, \qquad U(0)=\|u_0\|_\infty.
\end{equation}

\begin{rem}
  Though in the case $p<1$ there is in general no uniqueness, and therefore no comparison (the reaction is not Lipschitz), we always can compare with a supersolution which is a maximal solution of the equation, like the function $U$ in \eqref{flat} is, see \cite{dPV91}.
\end{rem}

In order to complete the proof of the first item in Theorem~\ref{main-teo}
we  observe that all the solutions have grow-up if $p\le1$.

\begin{lema}\label{lem-global-growup}
Let $u$ be a solution of \eqref{1.1}. If $p\le1$  then
$$
u(x,t)\to\infty
$$
uniformly in compact sets.
\end{lema}
\begin{proof}
We only note that this occurs for the solutions to the problem  if the reaction is localized in a bounded interval, $a(x)=\mathds{1}_{(-L,L)}$, see \cite{FdP}, and any solution to that problem (translated) is a subsolution to our problem.
\end{proof}

We now show that for $p>1$ there exist solutions that blow up in finite time provided the initial value is large in some sense.

\begin{lema}
If $p>\max\{m,1\}$ problem \eqref{1.1} has blow-up solutions.
\end{lema}
\begin{proof}
We observe that $u$ is a supersolution to the Dirichlet problem
$$
\left\{
\begin{array}{ll}
w_t=(w^m)_{xx}+w^p, \qquad &x\in (A,B),\, t>0\\
w(A,t)=w(B,t)=0,\\
w(x,0)=w_0(x),
\end{array}\right.
$$
for any interval $(A,B)\subset (0,\infty)$. Use then the results in \cite{SamarskiiBook87}.
\end{proof}

\begin{lema}\label{bump<m}
If $1<p\le m$ there exist blow-up solutions.
\end{lema}
\begin{proof}
We construct a self-similar subsolution
$$
\underline u(x,t)=(T-t)^{-\alpha} f(\xi)\qquad \xi=x(T-t)^{-\beta},
$$
satisfying $u(0,t)=0$. The self-similar exponents are given by
$$
\alpha=\frac1{p-1},\qquad \beta=\frac{p-m}2 \alpha,
$$
and the self-similar profile satisfies
\begin{equation}\label{eq.profiles}
(f^m)''-\beta \xi f'+f^p-\alpha f=0, \qquad f(0)=0.
\end{equation}

Using $(f^m)'(0)=\mu$ as shooting parameter we claim that there exists some  $\mu_0>0$ such that the corresponding profile $f_0$ satisfies
$$
f_0(\xi)>0 \quad \mbox{in } (0,\xi_0)\quad \mbox{and } \quad f_0(\xi_0)=0,
$$
for some $\xi_0>0$. This gives the desired blow-up subsolution with profile
$$
f(\xi)= \left\{
\begin{array}{ll}
f_0(\xi), \quad&\xi\in(0,\xi_0) \\
0, &\mbox{otherwise}
\end{array}\right.
$$
Then, if $u_0(x)>\underline u(x,0)$ the solution of \eqref{1.1} blows up.

In order to prove the claim we argue by contradiction, assuming that for every large $\mu$ the corresponding profiles $f_\mu$ are positive in $(0,\infty)$. Given any of such profiles with $\mu>1$ we take
$k=\mu^{\frac{p+m}2}$  and consider the function
$$
g_k(\xi)=\frac1{k^m} f^m(k^{\frac{m-p}2}\xi).
$$
It satisfies the initial value problem
$$
\left\{
\begin{array}{l}
g_k''+g_k^{p/m}-k^{1-p} \left( \beta\xi (g_k^{1/m})'-\alpha g_k^{1/m}\right)=0,\qquad\xi>0, \\
g_k(0)=0,\\
g_k'(0)=1.
\end{array}\right.
$$

We define the energy of the system at a point $\xi$ as
$$
E(\xi)=\frac12(g_k')^2+ V(g_k), \qquad V(s)= \frac{m}{p+m}s^{\frac{p+m}m}-\frac{\alpha m}{1+m} k^{1-p} s^{\frac{m+1}m}.
$$
Multiplying the equation by $g_k'$ we get that
$$
E'(\xi)=\frac\beta m k^{1-p}\xi g_k^{\frac{1-m}m} (g_k')^2\le0,
$$
since $\beta\le0$.  Thus,
$$
E(\xi)\le E(0)=\frac12.
$$
Also, calculating the minimum of the  potential $V$ we have
$$
E(\xi)\ge-ck^{-(p+m)}\ge -c.
$$
Since $p>1$ this implies that there exists two constants $C_1,C_2$ depending on $m$ and $p$  such that
$$
0\le g_k\le C_1,\qquad |g_k'|\le C_2.
$$
Hence, letting $k\to\infty$ we have that $g_k$ converges  uniformly in compact sets to a non negative function $G$. It is clear that $G$ satisfies
%
$$
\left\{
\begin{array}{l}
G''+G^{p/m}=0,\qquad \xi>0,\\
G(0)=0,\\
G'(0)=1.
\end{array}\right.
$$
However the solution of the above problem  crosses the axis at some finite point with non-zero slope. This is a contradiction and the claim is proved.
\end{proof}

\section{Fujita exponent}\label{sect-fujitaexp}

In this section we prove that the Fujita exponent is $p_c=m+2$, that is, all solutions blow up if $1<p\le m+2$, and if $p>m+2$ not all solutions do so.
In this last range $p>m+2$, it is easy to see that small initial data produce global solutions, by comparison with the global supersolutions corresponding to the case $a(x)=1$, see for instance the book \cite{SamarskiiBook87}. In fact they tend to zero for $t\to\infty$.

We divide the proof of blow-up below $p_c$ in three cases, $1<p\le m$, $m<p<m+2$ and $p=m+2$, the most difficult case being the last one.

\begin{lema}
If $1<p\le m$ then all solutions blow up in finite time.
\end{lema}
\begin{proof}
We only have to check that the self-similar subsolution constructed in Lemma \ref{bump<m} can be put below any solution if we let pass enough time.
\begin{enumerate}
\item It is clear when $p<m$ that we can do it since $\underline u(x,0)$ is small taking $T$ large, as well as its support is small, due to the fact that $\beta<0$.
\item For $p=m$ we note that $\underline u(x,0)$ is still small if $T$ is large but it has a fixed support $[0,\xi_0]$ since $\beta=0$. Nevertheless, using the penetration property of the solutions of the porous medium equation we obtain  that there exists $t_0>0$ such that the support of $u(\cdot,t_0)$ contains any interval.
\end{enumerate}
\end{proof}

\begin{lema}\label{p<m+2}
If $m<p<m+2$ then all solutions blow up in finite time.
\end{lema}
\begin{proof}
The proof is the same as for the global reaction and is an easy consequence of the energy argument of \cite{LevineSacks84}, also called concavity argument. In fact, defining the energy of a function $v$ as
 \begin{equation}\label{energy}
 E_v(t)=\dfrac12\int_{-\infty}^\infty|(v^m)_x|^2-\frac m{p+m}\int_0^\infty v^{p+m},
 \end{equation}
 we have that if for a solution $u$ to \eqref{1.1} there exists some $t_0$ such that $E_u(t_0)<0$ then $u$ blows up in finite time. Now we consider the Barenblatt function
\begin{equation}\label{barenblatt}
 B(x,t;D) = t^{-\frac1{m+1}}\left(
D - kx^2t^{-\frac2{m+1}}\right)_+^{\frac1{m-1}},
 \end{equation}
 where $k =\frac{m-1}{2m(m+1)}$, $D > 0$. It is a subsolution to our equation and it satisfies, for some constants $c_1,c_2$ depending only on $m$, $p$ and $D$,
$$
E_B(t)=c_1t^{-\frac{2m+1}{m+1}} - c_2t^{-\frac{p+m-1}{m+1}},
$$
which is negative for $t$ large provided $p<m+2$. The final step is a standard comparison argument: we make $B(x,1;D)$ small  by  taking $D$ small, so that it can be put below $u_0$; this implies $u(x,t)\ge B(x,t+1;D)$ for $t>0$; let $t_1$ be such that $E_B(t_1)<0$; let $v$ be the solution corresponding to the initial value $B(\cdot,t_1;D)$, which by the above energy argument blows up in finite time; since $u\ge v$ so does $u$. In the case $m<1$ we need the behaviour at infinity of every solution, see \cite{HP}, since the function \eqref{barenblatt}  is positive, while for $m=1$ a Gaussian is used instead of a Barenblatt function.
\end{proof}

We observe that the fact that the integral in the reaction term is performed only in $(0,\infty)$ does not affect  the original argument. In \cite{FdPV06} we used the fact that the integral in $(0,L)$ produces a different time power term if $L$ is finite, and so the Fujita exponent is  different in that case.

\begin{lema}
  If $p=m+2$ then all solutions blow up in finite time.
\end{lema}
\begin{proof}
We use the method introduced in \cite{Galaktionov94} to prove blow-up for the critical exponent in the case $a(x)=1$, but here the nonsymmetry of the problem makes things more involved.  The argument goes like this: assuming by contradiction that the solution is global,  we rescale and pass to the limit in time, thus obtaining a solution to some problem for which we prove nonexistence. 

Let $u$ be a global solution, and let $t_0\ge1$ and $D$ be such that $u(x,t_0)\ge B(x,t_0;D)$, where $B$ is given by \eqref{barenblatt} (if $m\ne1$, for $m=1$ we use instead a Gaussian like in the proof of Lemma \ref{p<m+2}).  We define the rescaled function
$$
v(\xi,\tau)=t^\alpha u(x,t),\qquad \xi=xt^{-\alpha},\quad\tau=\log t,\quad\alpha=\frac1{m+1}.
$$
We have that $v$ is a solution, for $\tau>\tau_0=\log t_0$, of the equation
\begin{equation}\label{v-crit}
v_\tau=(v^m)_{\xi\xi}+\alpha(\xi v)_\xi+a(\xi)v^{m+2}.
\end{equation}
If $g$ is the solution to equation \eqref{v-crit} with $g(\xi,\tau_0)=B(\xi,1;D)$, by comparison we have that $v\ge g$ for every $\tau>\tau_0$, and in particular $g$ is globally defined in $\tau$. For the special form of the initial value, it is easy to see that $g$ is nondecreasing in $\tau$, and therefore there exists the limit
$$
\lim_{\tau\to\infty}g(\xi,\tau)=f(\xi)\in[0,\infty].
$$
We claim the following alternative:

$a)$ $f$ is locally bounded. Thus we can pass to the limit in \eqref{v-crit}, by means of a Lyapunov functional,  to get that $f$ is a positive solution of
\begin{equation}\label{eq.perfil}
(f^m)''+\alpha (\xi f)'+\rho(\xi) f^{m+2}=0 \qquad \xi\in \mathbb{R},
\end{equation}
see \cite{Galaktionov94}.
Now we observe that the function
$$
\mathcal{E}(\xi)=(f^m)'+\alpha \xi f
$$
satisfies $\mathcal{E}'(\xi)=-\rho(\xi)f^{m+2}(\xi)$, so it is constant for $\xi<0$ and decreasing for $\xi>0$. Then, if we assume   $\mathcal E(0)=\mathcal E_0> 0$, we have that
$$
(f^m)'(\xi)\ge \mathcal E(\xi) \ge \mathcal E_0, \qquad \xi<0.
$$
This implies that there exists a point $\xi_1<0$ such that $f(\xi_1)=0$ and $(f^m)'(\xi_1)\ne0$. Therefore $\mathcal E(0)\le 0$, and there exists some $\xi_2>0$ such that $\mathcal E(\xi_2)=\mathcal E_2<0$. Exactly as before
$$
(f^m)'(\xi)\le \mathcal E(\xi) \le \mathcal E_2, \qquad \xi>\xi_2,
$$
so there exists a point $\xi_3>\xi_2$ such that $f(\xi_3)=0$, $(f^m)'(\xi_3)\ne0$. This gives a contradiction and $f$ cannot exist.

\

$b)$ There exists $\xi_0$ such that $f(\xi_0)=\infty$. Then $g$ is large in a nontrivial interval and this would imply that it blows up in a finite time. This is again a contradiction, and the theorem would be proved.

We have that $f$ satisfies equation \eqref{eq.perfil} in any interval in which it is bounded. It is clear that $f$ cannot have any minima since at such a point we would have from the equation $(f^m)''<0$. This implies
\begin{equation}\label{limsup}
\lim\limits_{\xi\to\xi_0^-}f(\xi)=\limsup\limits_{\xi\to\xi_0^-}(f^m)'(\xi)=\infty.
\end{equation}

Assume $\xi_0>0$. If $f$ is bounded in some interval $(\xi_0-\delta,\xi_0)$, $\delta\le\xi_0$, then $f$ is increasing in that interval with
$$
(f^m)'(\xi)\le\mathcal E(\xi_0-\delta/2),\qquad \xi_0-\delta/2<\xi<\xi_0.
$$
This is a contradiction and thus $f(\xi)=\infty$ for every $0\le\xi\le\xi_0$.
Moreover, if $f$ is bounded in  $\xi<0$, we have
$$
(f^m)'(\xi)+\alpha\xi f(\xi)=c<0,
$$
by the above. Thus by \eqref{limsup}, there is a sequence $\xi_j\to0^-$ such that $|\xi_j|f(\xi_j)\to\infty$. The same argument works from the left to the right, assuming $\xi_0<0$. In conclusion $f$ is large in some interval $|\xi|\le\xi_*$, that could be small, but it satisfies that $\xi_*f(\xi_*)$ is large.

Let us now show that in this situation the function $g$ blows up in finite time.
By  the monotonicity of $g$ in time we have that for any large constant $A_*>0$ there exists $M>0$, $\xi_M>0$ and $\tau_M$ such that $M\xi_M \ge A_*^{3/2}$ and $g(\xi,\tau)\ge M $ for every $|\xi|\le \xi_M$, $\tau\ge\tau_M$. Now we argue as in \cite{FdPV06}. Let $z(x,t)=e^{-\alpha\tau} g(\xi,\tau)$ be the function $g$ in the original variables, and define $h(x,t+e^{\tau_M})$ the solution
of \eqref{1.1} with initial datum
$$
W(x)=\lambda^{-1}(A-\lambda^{-2}x^2)_+
$$
where
$$
A=(\xi_M M)^{2/3}, \qquad \lambda=e^{\alpha\tau_M} (\xi_M^2/M)^{1/3}.
$$
It is clear that $W(x)\le z(x,e^{\tau_M})$, since
$$
\begin{array}{l}
z(x,e^{\tau_M})=e^{-\alpha\tau_M} g(\xi,\tau_M)\ge e^{-\alpha\tau_M}M \qquad\text{for } |x|\le e^{-\alpha\tau_M}\xi_M, \\ [2mm]
W(x)\le W(0)=\lambda^{-1}A=e^{-\alpha\tau_M}M, \\ [2mm]
supp(W)=\{|x|\le \lambda A^{1/2}\}=\{|x|\le e^{-\alpha\tau_M}\xi_M\}.
\end{array}
$$
Moreover,
$$
E_h(0)=\lambda^{-(2m+1)} A^{2m+1/2} \Big( c_1-c_2 A^2\Big),
$$
for some $c_1,c_2$ depending only on $m$. This is negative for $A>A_*=A_*(m)$.
Thus $h$ blows up in finite time, and by comparison $z$, or which is the same $g$, also blows up. This ends the proof.
\end{proof}


\section{Grow-up rates}\label{sect-guprates}

The aim of this section is to study  the speed  at which the global unbounded (grow-up) solutions to problem (1.1) tend to infinity. We therefore consider the range $p\le1$. In order to avoid nonuniqueness issues when $p<1$ we assume in that case that the initial value is positive for $x>0$, that is where the non-Lipschitz reaction applies.

As we have said in the Introduction, the upper estimate of the grow-up rate is given by comparison with the function in \eqref{flat}. In the case of global reaction $a(x)=1$ this is sharp if $p<1$ or $m<1$. In fact we have for $t$ large
\begin{equation}\label{decay-a=1}
u(x,t)\sim \left\{
\begin{array}{ll}
t^{\frac1{1-p}}, \qquad &p<1, \\
 e^t \qquad &m<1=p,\\
t^{-\frac12} e^t, & m=1=p,\\
e^{\frac2{m+1} t}, & m>1=p,
\end{array}\right.
\end{equation}
see \cite{AE86,dPV90,vazquez}.

On the other hand, when $a(x)=\mathds{1}_{(-L,L)}$ the rates are proved in \cite{FdP}. Though in that situation the global existence exponent is different, $p_0=\max\{1,\frac{m+1}2\}$, we quote the results proved in \cite{FdP} in our range $p\le1$:

$i)$ if $p\le1<m$  then
\begin{equation}
u(x,t)\sim t^{\frac1{m+1-2p}},
\end{equation}
in compact sets.

$ii)$  if $m<p<1$ then
\begin{equation}\label{gup-rate-a=L}
u(x,t)\sim
\left\{
\begin{array}{ll}
t^{\frac{1}{1-p}},\qquad & \text{for } |x|<L,\\
t^{\frac1{1-m}},&  \text{for } |x|>L,
\end{array}\right.
\end{equation}
provided that the initial datum satisfies
\begin{equation}\label{eq.comportamiento.dato}
|x|^2 u_0^{1-m}(x)\sim 1.
\end{equation}

$iii)$ if  $m<p=1$ then
\begin{equation}\label{gup-rate-a=L_p1}
u(x,t)\sim
\left\{\begin{array}{ll}
e^t,\qquad & \text{for } |x|<L,\\
t^{\frac1{1-m}},&  \text{for } |x|>L,
\end{array}\right.
\end{equation}
provided that the initial datum satisfies
\begin{equation}\label{eq.comportamiento.dato_p1}
  |x|^2 u_0^{1-m}(x)\sim \log(x).
\end{equation}

$iv)$ if $p=m=1$ then
\begin{equation}
\lim_{t\to\infty}\frac{\log u(x,t)}{t}=\omega(L)\in(0,1).
\end{equation}

\

We prove in this paper that for problem~\eqref{1.1} the rate can be that corresponding to global reaction or to reaction localized in a bounded interval, or none of them, depending on the sign of $p-m$. We can also have a different rate inside or outside the region where the reaction applies when $p>m$, like in the case $a(x)=\mathds{1}_{(-L,L)}$.

\subsection{\underline{Case $p= 1$}}

\

Though the reaction is linear this is the more involved case. We consider separately  the three cases according to $m$ being larger, equal or smaller than~1.

The proof of the grow-up rate follows by comparison with special selfsimilar subsolutions and supersolutions.
We construct such functions in the form
\begin{equation}\label{selsimlar_p1}
w(x,t)=e^{\alpha t} f(x e^{-\beta t}),
\end{equation}
where necessarily
$$
\beta=\frac{m-1}{2}\alpha.
$$
Also, by \eqref{decay-a=1} we consider only $\alpha\le 2/(m+1)$.

The profile $f$ will be given by matching  two functions,
\begin{equation}\label{profile-f}
f(\xi)=\left\{
\begin{array}{ll}
\psi(\xi),\qquad &\xi\ge0,\\
\phi(-\xi),\qquad &\xi\le0,
\end{array}\right.
\end{equation}
where $\psi$ and $\phi$ are the truncation by zero of the solutions of the initial value problems, for some $\lambda\in\mathbb{R}$,
\begin{equation}\label{psi}
\left\{
\begin{array}{l}
(\psi^m)''+\beta \xi \psi'+(1-\alpha) \psi= 0,\qquad \xi>0,\\
\psi(0)=1,\\
\psi'(0)=\lambda,
\end{array}\right.
\end{equation}
\begin{equation}\label{phi}
\left\{
\begin{array}{l}
(\phi^m)''+\beta \xi \phi'-\alpha \phi= 0,\qquad \xi>0,\\
\phi(0)=1,\\
\phi'(0)=-\lambda.
\end{array}\right.
\end{equation}

We start with $m>1=p$.

\subsubsection{\underline{Slow diffusion, $m>1$}.}\label{p=1<m}
The existence of solutions with compact support for equations  of the above type  has been studied in \cite{GP}. Let us consider, as in that  paper, the problem for some $\xi_0>0$ given,
\begin{equation}\label{GP}
\left\{
\begin{array}{ll}
(g^m)''+\beta \xi g'-q g=0, \qquad & \xi<\xi_0,\\
g(\xi_0)=(g^m)'(\xi_0)=0.
\end{array}\right.
\end{equation}
It is proved in \cite{GP},
\begin{teo}\label{fun-g} Let $\beta>0$. There exists a continuous solution $g$ to problem~\eqref{GP} such that $g(0)>0$ for $2\beta+q>0$; $g(0)=0$ for $2\beta+q=0$; and if $2\beta+q<0$ there exists a point $\xi_1\in(0,\xi_0)$ with $g(\xi_1)=0$. Moreover, in the first case, $g'(0)<0$ if $\beta+q>0$; $g'(0)=0$ if $\beta+q=0$; and $g'(0)>0$ if $\beta+q>0$. Finally
\begin{equation}\label{xi0}
g(\xi)\sim(\xi_0-\xi)^{\frac1{m-1}}\qquad\text{for }\; \xi\to\xi_0^-.
\end{equation}
\end{teo}

Translating this result to our problems \eqref{psi} and \eqref{phi}, where $q$ takes the values, respectively, $q=\alpha-1<0$ and $q=\alpha>0$, we obtain the following results.

\begin{cor}\label{cor-solucion-soporte-compacto}
\begin{enumerate}
\item For each $\alpha>0$ there exists a unique $\lambda_-(\alpha)>0$ such that problem \eqref{phi} with $\lambda=\lambda_-(\alpha)$ has a decreasing solution with compact support.

\item Problem \eqref{psi} has solutions with compact support for some $\lambda$ if and only if $\alpha>1/m$, the solution being unique for each $\alpha$  given, and thus $\lambda=\lambda_+(\alpha)$. Moreover, $\lambda_+(\alpha)>0$ if $\alpha<2/(m+1)$; and $\lambda_+(2/(m+1))=0$.
\end{enumerate}
\end{cor}

If we find some $\alpha\in(1/m,2/(m+1))$ such that  $\lambda_-(\alpha)=\lambda_+(\alpha)$, we will obtain a solution $w$ with frofile $f$ defined in $\mathbb{R}$ which has compact support. But we are also interested in subsolutions, and these are obtained constructing profiles with compact support $[-a,b]$ with a bad behaviour at the interfaces $(f^m)'(-a)>0$, $(f^m)'(b)<0$. On the other hand, positive profiles will serve as supersolutions.

Thus, in order to study in more detail the solutions to the equation in \eqref{GP} we introduce the variables
\begin{equation}\label{variable}
X=\frac{\xi g'}{g},\qquad Y=\frac1m \xi^2 g^{1-m},\qquad \eta=\log\xi.
\end{equation}
We also fix the value $g(0)=1$ and consider the different values of $g'(0)$. We obtain the differential system,
$$
\left\{
\begin{array}{l}
\dot{X}=X(1-mX)+Y(q-\beta X),\\
\dot{Y}=Y(2-(m-1)X),
\end{array}\right.
$$
defined in the half-plane $Y\ge0$, where $\dot{X}=dX/d\eta$. As we have said only the values  $q=\alpha$ and $q=\alpha-1$ are of interest, with $\alpha\in(1/m,2/(m+1))$. We have two finite critical points
$$
P_1=(0,0),\qquad P_2=(1/m,0)
$$
(if $q=\alpha-1$ there exists a third critical point but it lies in the lower half-plane), and three critical points at infinity
$$
\Lambda_1=(-\infty,\infty),\qquad \Lambda_2=(-\frac{q}{\beta},\infty),\qquad\Lambda_3=(\infty,\infty).
$$

The point $P_1$ is an unstable node: we have a trajectory $\Gamma_0$ escaping this point from $\eta=-\infty$ along the vector $(q,1)$, and a family of trajectories $\Gamma_\kappa$, $\kappa\neq0$, behaving near the origin like
\begin{equation}\label{cerca_cero}
X\sim \kappa \sqrt Y.
\end{equation}
The first one produces a profile $g$ with $g'(0)=0$. The profile corresponding to each $\Gamma_\kappa$ satisfies $g'(0)=\kappa/\sqrt m$. The point $P_2$ is a saddle and plays no role at this stage.

Now fix $q=\alpha-1<0$.
We first observe that defining the energy associated to the problem
$$
E_g(\xi)=\frac12((g^m)')^2+\frac{1-\alpha}{1+m}g^{1+m},
$$
it satisfies
$$
E_g'(\xi) =-\beta m \xi g^{m-1} (g)'^2\le0.
$$
Therefore $g$ is bounded, and all the trajectories starting at $P_1$ must go to one the points at infinity $\Lambda_1$ or $\Lambda_2$. In fact $\Lambda_3$ is unstable.

The profiles satisfying \eqref{xi0} correspond to trajectories entering the point $\Lambda_1$ linearly, since they satisfy
$$
\lim_{\xi\to\xi_0}X(\xi)=-\infty,\quad\lim_{\xi\to\xi_0}Y(\xi)=\infty,\quad
\lim_{\xi\to\xi_0}\frac{Y(\xi)}{X(\xi)}=-D<0.
$$

Using  the equation
\begin{equation}\label{dYdX}
\frac{dY}{dX}=\frac{Y(2-(m-1)X)}{X(1-mX)+Y(q-\beta X)},
\end{equation}
we get that the only possible behaviours near $\Lambda_1$ are
$$
Y\sim -X,\qquad Y\sim |X|^{\frac{m-1}m}.
$$

Thus from Corollary \ref{cor-solucion-soporte-compacto} we get  that the trajectory $\Gamma_{\kappa^+}$ with $\kappa^+=\lambda_+(\alpha)\sqrt m$ joins  $P_1$ with $\Lambda_1$  satisfying $Y\sim-DX$ near $\Lambda_1$, and it is the unique trajectory with that behaviour at infinity.
All the other trajectories joining these two points enter $\Lambda_1$  below $\Gamma_{\kappa^+}$, and above this trajectory near the origin. This implies that the corresponding profiles have slopes at the origin $g'(0)<\lambda_+(\alpha)$. Observe  that this implies that $g$ vanishes at some point $b<\infty$ with $g(\xi)\sim(b-\xi)^{\frac1m}$, so $(g^m)'(b)<0$.

On the other hand, the trajectories with $\kappa>\kappa^+$ must go to $\Lambda_2$. The corresponding profiles are positive with
$$
g(\xi)\sim
\xi^{-\frac{1-\alpha}\beta}
$$
for $\xi$ large.

\begin{figure}[!ht]\begin{center}
\includegraphics[scale=.4]{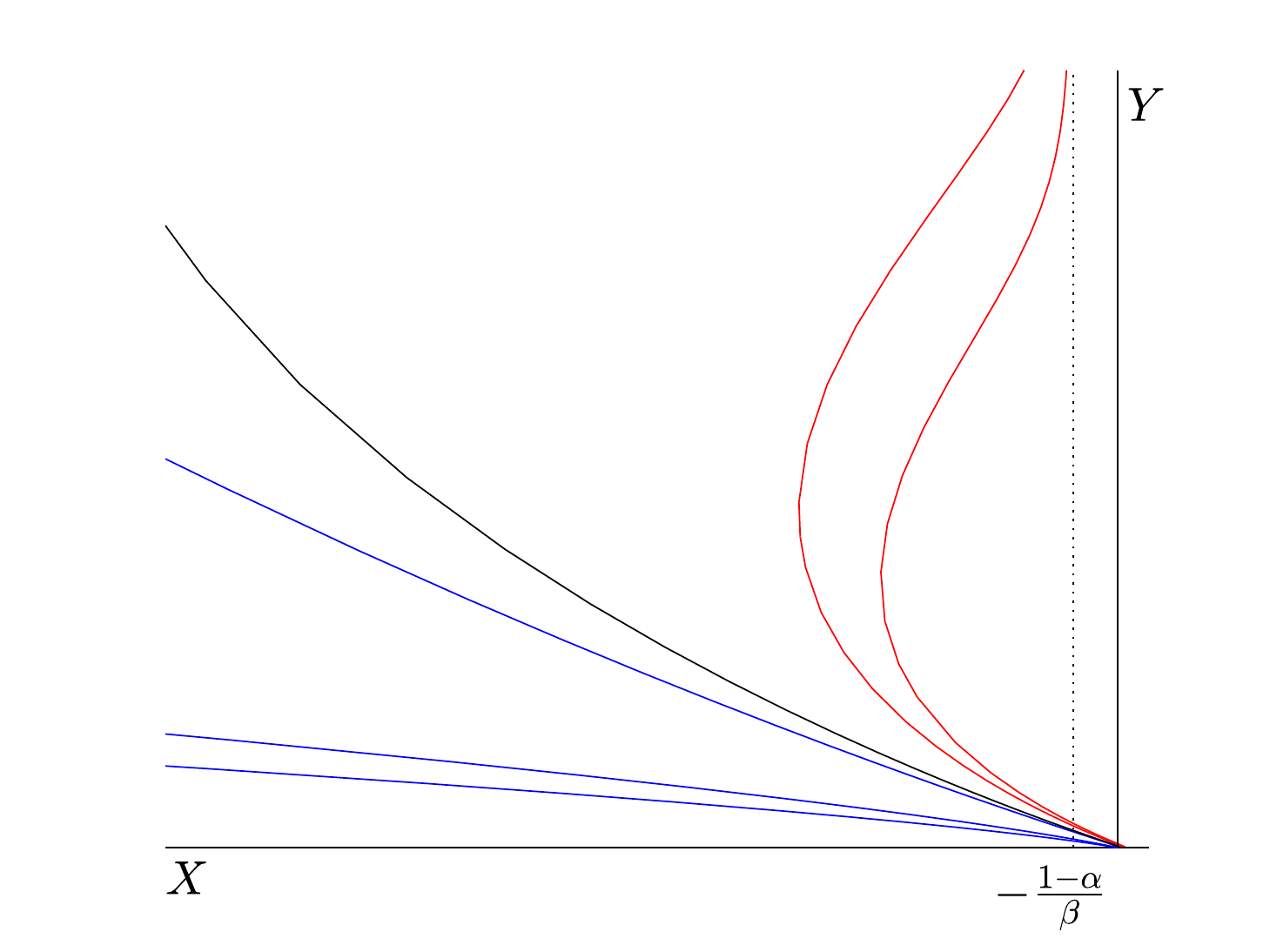}
\includegraphics[scale=.4]{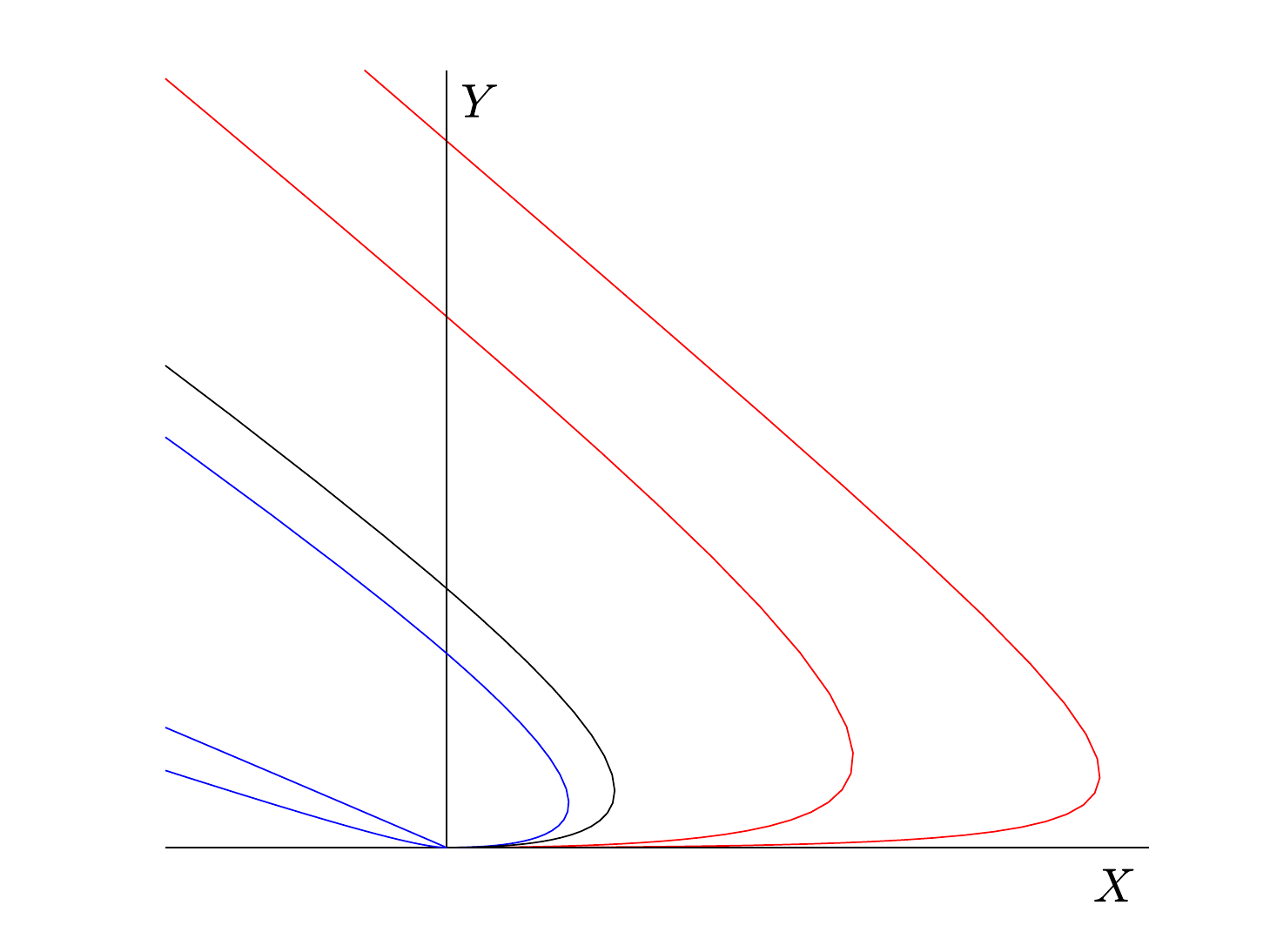}
\caption{Trajectories in the phase-plane for $q=\alpha-1$ (on the right a zoom at the origin). The black line enters $\Lambda_1$ linearly; the blue lines go to $\Lambda_1$ like $Y\sim |X|^{\frac{m-1}m}$; the red lines go to $\Lambda_2$.}
\label{fig.gup}
\end{center}
\end{figure}

In summary we have proved the following result.

\begin{lema}\label{lem-psi} Let $\psi$ be a solution of \eqref{psi} with some $\alpha\in(1/m,2/(m+1))$. There exists some $\lambda_+=\lambda_+(\alpha)>0$ such that
\begin{enumerate}
\item If $\lambda<\lambda_+$ there exists $\xi_0<\infty$ such that $\psi(\xi_0)=0>(\psi^m)'(\xi_0)$.
\item If $\lambda=\lambda_+$ there exists $\xi_0<\infty$ such that $\psi(\xi_0)=(\psi^m)'(\xi_0)=0$.
\item If $\lambda>\lambda_+$ the solution $\psi$ is positive  and $\psi\sim
\xi^{-\frac{1-\alpha}\beta}$ for $\xi$ large.
\end{enumerate}
\end{lema}

We comment by passing what is the behaviour in the case $\alpha=1/m$. If we trace back the unique trajectory entering $\Lambda_1$ linearly, we see that it goes to $P_2$. In fact integrating the equation between $\xi$ and $a$ we get
$$
-(g^m)'(\xi)=\beta \xi g(\xi)-\beta \int_\xi^a g(s)ds.
$$
A second integration gives
$$
g^m(\xi)=\beta \xi \int_\xi^ag(s)ds.
$$
Then $\lim_{\xi\to0} Y(\xi) =0$ trivially, while
$$
\lim_{\xi\to0} X(\xi)=\lim_{\xi\to0}\frac{\xi(g^m)'(\xi)}{mg^m(\xi)}=
\frac1m\lim_{\xi\to0}\frac{\xi\left(-\beta \xi g(\xi)+\beta \int_\xi^a g(s)ds\right)}{\beta \xi \int_\xi^ag(s)ds}=\frac1m.
$$
Therefore all the trajectories starting at $P_1$ must go to $\Lambda_1$ like $Y\sim |X|^{\frac{m-1}m}$. We obtain profiles with a bad interface behaviour for every value of $\lambda$.

\

The phase-space \eqref{variable} is studied in the same way in the case $q=\alpha>0$.
Following the same argument for problem \eqref{phi} we obtain:

\begin{lema}\label{lem-phi} Let $\phi$ be a solution of \eqref{phi} with some $\alpha\in(0,2/(m+1))$. There exist  $\lambda_-=\lambda_-(\alpha)>0$ such that:
\begin{enumerate}
\item If $\lambda>\lambda_-$ there exists $\xi_1<\infty$ such that $\phi(\xi_1)=0>(\phi^m)'(\xi_1)$.
\item If $\lambda=\lambda_-$ there exists $\xi_1<\infty$ such that $\phi(\xi_1)=(\phi^m)'(\xi_1)=0$.
\item If $\lambda<\lambda_-$ the solution $\phi$ is positive and unbounded.
\end{enumerate}
Moreover $\lambda_-(\alpha)=K /\sqrt{\alpha}$.
\end{lema}

We now study the matching.

\begin{lema}\label{lema-alpha*}
There exists a unique $\alpha_*\in(1/m,2/(m+1))$ such that $\lambda_-(\alpha_*)=\lambda_+(\alpha_*)$.
\end{lema}
\begin{proof}
Define the continuous function $h(\alpha)=\lambda_+(\alpha)-\lambda_-(\alpha)$. It is clear that $h(2/(m+1))<0$. On the other hand,
taking $\lambda=\lambda_-(1/m)$ in \eqref{phi} and \eqref{psi} we obtain a profile which crosses the axis at some positive point with bad interface behaviour. Therefore by continuous dependence of the profile with respect to the parameter $\alpha$ we have the same behaviour for $\alpha=\varepsilon+1/m$. This implies $h(\varepsilon+1/m)>0$. Then there exists $\alpha_*\in(1/m,2/(m+1))$ with $h(\alpha_*)=0$. The uniqueness follows by comparison. Indeed, if we assume that $h(\alpha_1)=0=h(\alpha_2)$ with $\alpha_1<\alpha_2$ we have that the solutions $w_1$ with profile $f_1,f_2$ given in \eqref{profile-f} with $\alpha=\alpha_1,\alpha_2$ satisfy
$$
w_1(x,0)=f_1(x)>e^{-\alpha_2 t}f_2(xe^{\beta_2 t})=w_2(x,-t_1)
$$
for some $t_1>0$.
This implies  $w_1(x,t)\ge w_2(x,t-t_1)$ for any $t>0$. In particular at  $x=0$ this means $f_1(0)e^{\alpha_1 t}\ge f_2(0)e^{\alpha_2 (t-t_1)}$, which is impossible if  $t$ is large.
\end{proof}

\begin{teo} Let $u$ be the solution to problem \eqref{1.1} with $p=1<m$.
Let $\alpha_*$ be given in Lemma \ref{lema-alpha*} and define $\gamma_*=\frac{2(1-\alpha_*)}{(m-1)\alpha_*}$.
\begin{enumerate}
\item If there exists some $1<\gamma<\gamma_*$ such that $u_0(x) \sim x^{-\gamma}$ as $x\to\infty$, then
$$
C_1 e^{\alpha(\gamma) t}\le u(x,t)\le C_2 e^{\alpha(\gamma) t},
$$
where
$$
\alpha(\gamma)=\frac{2}{2+\gamma(m-1)}.
$$
\item If $\limsup_{x\to\infty}x^{-\gamma_*}u_0(x)<\infty$, then for all $\varepsilon>0$
$$
C_1 e^{\alpha_* t}\le u(x,t)\le C(\varepsilon) e^{(\alpha_*+\varepsilon) t}.
$$
\item If $u_0$ has support bounded from the right then
$$
C_1 e^{\alpha_* t}\le u(x,t)\le C_2 e^{\alpha_* t}.
$$
The above estimates are uniform in compact subsets of $\mathbb R$ for $t$ large.
\end{enumerate}
\end{teo}
\begin{proof}
The proof follows by comparison with the self-similar functions constructed before.  We define:

$\bullet$ $w_*$ the self-similar function given in \eqref{selsimlar_p1} with $\alpha=\alpha_*$. It is a solution to \eqref{1.1} with compact support
$$
\mbox{supp}(w(\cdot,t))=[-K_- e^{\beta t},K_+ e^{\beta t}].
$$

$\bullet$ For $\alpha\in(\alpha_*,2/(m+1))$, which implies $\lambda_+(\alpha)<\lambda_-(\alpha)$, we consider $w_-$ the self-similar function given in \eqref{selsimlar_p1} with $\lambda=\lambda_-(\alpha)$. It is a solution to \eqref{1.1} with support bounded from the left
$$
\mbox{supp}(w(\cdot,t))=[-K_1 e^{\beta t},\infty),\qquad \lim_{x\to\infty}x^{\frac{2(1-\alpha)}{(m-1)\alpha}}w(x,t)=K_2 e^t.
$$

In this case, we also consider $w_\lambda$ the self-similar function given in \eqref{selsimlar_p1} with $\lambda_+(\alpha)<\lambda<\lambda_-(\alpha)$. It is a positive solution to \eqref{1.1} such that
$$
\lim_{x\to\infty}x^{\frac{2(1-\alpha)}{(m-1)\alpha}}w_\lambda(x,t)=K_2 e^t, \qquad \lim_{x\to-\infty}w_\lambda(x,t)= \infty.
$$
We now consider the different cases in the statement of the theorem.

\begin{enumerate}
\item $u_0(x)\sim x^{-\gamma}$  with $1<\gamma<\gamma_*$. Taking $\alpha=\alpha(\gamma)\in(\alpha_*,2/(m+1))$ we have $\lambda_+(\alpha)<\lambda_-(\alpha)$, and the functions $u_0$, $w_-$ and $w_\lambda$ have the same behaviour at infinity. Then there exists $t_1$ large  enough such that
    $$
    w_-(x,-t_1)\le u_0(x) \le w_{\lambda}(x,t_1),
    $$
    and by comparison
    $$
    w_-(x,t-t_1)\le u(x,t) \le w_{\lambda}(x,t+t_1).
    $$
    The grow-up rate follows.
\item $u_0(x)\le x^{-\gamma_*}$. The lower bound follows by comparison with $w_*(x,t-t_1)$. For the upper bound we compare with $w_\lambda(x,t+t_1)$ with $\alpha=\alpha_*+\delta$, $\delta>0$ small.
\item $u_0$ with compact support. We compare from below as in the previous case, and from above with $w_*(x,t+t_1)$.
\end{enumerate}
\end{proof}

We observe that for any initial value $u_0$ the grow-up rate is always exponential, like for global reaction $a(x)=1$, but with an exponent strictly smaller $\alpha\le\alpha_*<2/(m+1)$. In the case  of a localized reaction, $a(x)=\mathds{1}_{(-L,L)}$, the grow-up was polynomial.

\

The second case to consider when $p=1$ is $m=1$, where things are more or less explicit.

\subsubsection{\underline{Linear diffusion $m=1$}}

\begin{lema}
Let $u$ be the solution to \eqref{1.1} with $p=m=1$. Then,
$$
C_\varepsilon e^{(1-\varepsilon)t}\le u(x,t)\le C_2 e^{t},
$$
uniformly in compact subsets of $\mathbb R$.
\end{lema}
\begin{proof}
The upper estimate is given by comparison with the function in \eqref{flat}. For the lower bound we use again comparison, this time with  an exponential selfsimilar  function, see \eqref{psi}, \eqref{phi}. Since here $\beta=0$, we look for a function in separated variables
$$
w(x,t)=e^{\alpha t }f(x), \qquad 0<\alpha<1,
$$
where the profile $f$ satisfies
$$
\left\{
\begin{array}{ll}
f''+(1-\alpha) f=0, \qquad &x>0,\\
f''-\alpha f=0,& x<0, \\
f(0)=1.
\end{array}\right.
$$
This gives
$$
f(x)=\left\{
\begin{array}{ll}
C_1 e^{\sqrt{\alpha} x}+C_2e^{-\sqrt{\alpha}x} , &x<0, \\
C_3\sin(\sqrt{1-\alpha}\ x)+\cos(\sqrt{1-\alpha} \ x),\quad &x>0.
\end{array}\right.
$$
The matching condition at $x=0$ means
$$
C_1+C_2=1, \qquad C_3=\sqrt{\frac\alpha{1-\alpha}} (C_1-C_2).
$$
Notice that for any $x_-<0$ given we can take
$$
C_2=\frac{e^{\sqrt\alpha x_-}}{e^{\sqrt\alpha x_-}-e^{-\sqrt\alpha x_-}},
$$
so that $f(x_-)=0$. Moreover, for
$$
x_+=\frac1{\sqrt{1-\alpha}}\arctan\left(\frac{-1}{C_3}\right)\in \left(\frac{\pi}{2\sqrt{1-\alpha}},
\frac{\pi}{\sqrt{1-\alpha}}\right),
$$
we have $f(x)>0$ in $(0,x_+)$ and $f(x_+)=0$. This profile gives us a subsolution by the procedure of truncation by zero. We denote by $f_{x_-}$ this truncated profile.

Let now $u$ be a solution of \eqref{1.1}. Since the heat equation has infinite speed of propagation, we can assume without loss of generality that $u_0(x)>0$. Then there exists $t_1>0$ such that
$$
u_0(x)\ge e^{-\alpha t_1} f_{x_-}(x).
$$
By comparison we deduce $u(x,t)\ge w_{x_-}(x,t-t_1)$. We obtain the lower grow-up rate for compact subsets of $(-\infty,\frac{\pi}{2\sqrt{1-\alpha}})$ and for every $\alpha<1$.

Finally we observe that for $A>0$ the function $w_{x_-}(x-A,t-t_1)$ is also a subsolution to \eqref{1.1}, so we obtain the lower grow-up rate for any compact subset of $\mathbb R$.
\end{proof}

\

We end by considering  the case $m<1=p$.

\subsubsection{\underline{Fast diffusion $m<1$}.}
Here the rate is different for $x>0$ and for $x<0$, as in the case of a localized reaction, $a(x)=\mathds{1}_{(-L,L)}$.

We first show that the grow-up rate given in \eqref{decay-a=1} is sharp for compact subset of $\mathbb R^+$ by proving the lower bound.  To do that we compare with a subsolution in separated variables with compact support in $\mathbb R^+$,
$$
\underline u(x,t)=f(x) g(t).
$$
Notice that since $u$ has global grow-up, see Lemma \ref{lem-global-growup}, we have $u(x,t_0)\ge \underline u(x,0)$ for $t_0$ large enough, so then  the comparison of the initial data is granted by a time shift.

\begin{lema}\label{lema-exp-lower}
Let $u$ be a solution of \eqref{1.1} with $m<1=p$. Then,
$$
u(x,t)\ge c e^t
$$
uniformly in compact subset of $\mathbb R^+$.
\end{lema}
\begin{proof}
Let $\phi$ be the solution to the problem
$$
\left\{
\begin{array}{l}
(\phi^m)''+ \phi=0,\quad \xi>0,\\
\phi(0)=0, \\
\phi'(0)=1.
\end{array}\right.
$$
Since $\phi^m$ is concave $\phi$ must vanish at some point $x_0<\infty$. Now we consider the rescaled function
$$
f(x)=A \phi(A^{\frac{1-m}2}x),
$$
which satisfies the same equation and vanishes at $x=x_0A^{-\frac{1-m}2}$. This is the spatial part of our subsolution. The time part $g$ is defined as the solution to
$$
\left\{
\begin{array}{l}
g'=g-g^m,\quad \xi>0,\\
g(0)>1.
\end{array}\right.
$$
We have $u(x,t+t_0)\ge f(x)g(t)$ for any $x>0$ and $t>0$. Since  $g'\sim g$ as $t\to\infty$, the comparison  gives the desired lower bound.
\end{proof}

In order to obtain the grow-up rate for $\mathbb R^-$, we note that by \eqref{decay-a=1} $u$ is a subsolution of the problem
$$
\left\{
\begin{array}{ll}
w_t=(w^m)_{xx},\qquad &x<0, t>0,\\
w(0,t)= C_1 e^t, & t>0,\\
w(x,0)=w_0(x), & x<0.
\end{array}\right.
$$
It is proved in \cite{FdP} that there exists a unique self-similar solution of exponencial type
$$
W(x,t)=e^t f(x e^{\frac{1-m}{2}t}),
$$
which is increasing in both variables $x$ and $t$. Moreover, for $|\xi|$ large
$$
f(\xi)\sim |\xi|^{\frac{-2}{1-m}}\left( \log|\xi|\right)^{\frac{1}{1-m}}.
$$
Then, if the initial datum satisfies
\begin{equation}\label{eq-condicion}
u_0(x)\sim |x|^{\frac{-2}{1-m}}\left( \log|x|\right)^{\frac{1}{1-m}},\qquad x\sim -\infty,
\end{equation}
we can take as a supersolution $\overline w (x,t)= A W(x,t)$. Notice that from the property $W_t\ge0$ we have
$$
\overline w_t -\overline w_{xx}=(A-A^m) w_t \ge 0,
$$
provided $A>1$. Moreover, taking $A$ large enough we get $\overline w(x,0)\ge u_0(x)$.

On the other hand, by Lemma \ref{lema-exp-lower} we have that $u$ is a supersolution to the problem
$$
\left\{
\begin{array}{ll}
w_t=(w^m)_{xx},\qquad &x<1, t>0,\\
w(1,t)= C_2 e^t, & t>0,\\
w(x,0)=w_0(x), & x<1.
\end{array}\right.
$$
and $\underline w (x,t)= A W(x,t)$ with $A$ small enough to have $\underline w(x,0)\le u_0(x)$ is a subsolution.

As a conclusion,
\begin{lema}
Let $u$ be a solution  of \eqref{1.1}  with $m<1=p$, such that the initial datum $u_0$ satisfies the condition \eqref{eq-condicion}. Then, for $x<0$
$$
u(x,t)\sim t^{\frac1{1-m}}.
$$
\end{lema}

\subsection{\underline{Case $p<1$}}

\

Here we distinguish between $m<p$ and $m\ge p$.

\subsubsection{\underline{Case $m<p$}.}

\begin{lema}
Let $m< p< 1$. If $u_0$ satisfies \eqref{eq.comportamiento.dato} for $x\sim -\infty$, then
$$
u(x,t)\sim \left\{
\begin{array}{ll}
t^{\frac1{1-p}}, \qquad &x>0, \\
t^{\frac1{1-m}},\qquad &x<0,
\end{array}\right.
$$
uniformly in compact sets.
\end{lema}
\begin{proof}
The proof follows in the same way as in the case $p=1$, using here the selfsimilar profile
$$
W(x,t)=t^{\frac{1}{1-p}} f(x t^{\frac{p-m}{2(1-p)}})
$$
constructed in \cite{FdP}, which is again increasing in both variables $x$ and $t$,  and that satisfies, for $|\xi|$ large,
$$
f(\xi)\sim |\xi|^{\frac{-2}{1-m}}.
$$
\end{proof}

\subsubsection{\underline{Case $m\ge p$}.}
\begin{lema}
Let $p<1\le m$. then
$$
u(x,t)\ge c t^{\frac1{1-p}}.
$$
uniformly in compact sets of $\mathbb R$.
\end{lema}
\begin{proof}
We consider a subsolution in selfsimilar form
$$
w(x,t)=t^{\alpha} f(\xi),\qquad \xi=xt^\beta,
$$
where
$$
\alpha=\frac1{1-p},\qquad \beta = -\frac{m-p}{2}\alpha,
$$
and the selfsimilar profile satisfies
$$
\mathfrak{L}(f):=(f^m)''-\beta\xi f'+a(\xi)f^p-\alpha f\ge 0.
$$
We construct the profile gluing four functions. Let $A>0$ be a constant to be fixed and put $\xi_0=-\sqrt{2/\alpha} A^{\frac{m-1}{2m}}$.
\begin{enumerate}
\item For $\xi\le\xi_0$ we put $f_1(\xi)=0$.
\item For $\xi_0\le\xi\le0$ we define
$$
f_2^m(\xi)=
A+ \sqrt{2\alpha} A^{\frac{1+m}{2m}}\xi+\alpha A^{1/m}\frac{\xi^2}{2}.
$$
Notice that $f^m_2(\xi_0)=(f_2^m)'(\xi_0)=0$. Moreover since $\beta\le0$ and $f_2$ is non-decreasing
$$
\mathfrak{L}(f_2)\ge (f_2^m)''-\alpha f_2\ge \alpha A^{\frac1m}-\alpha f_2(0)=0.
$$
\item For $0\le \xi\le \xi_1=\sqrt{2\alpha} A^{\frac{1+m-2p}{2m}}$  we define
$$
f_3^m(\xi)=A+ \sqrt{2\alpha} A^{\frac{1+m}{2m}}\xi - (A^{p/m}-\alpha A^{1/m}) \frac{\xi^2}{2}.
$$
We have  $(f_3^m)(0)=(f_2^m)(0)=A$, $(f_3^m)'(0)=(f_2^m)'(0)$, so this function $f_3$ matches well with $f_2$. Also $f_3$ is increasing in $0<\xi<\xi_1$, with
$$
f_3(\xi_1)=
A^{\frac1m} \left(1+\alpha \frac{A^{\frac{1-p}{m}}}{1-\alpha A^{\frac{1-p}{m}}}\right)^{\frac1m}, \qquad (f_3)'(\xi_1)=0.
$$
Since $p<1$ we get that for $A$ small enough, both $\xi_1$ and  $f_3(\xi_1)$ are small. Hence, the function $f_3^p(\xi)-\alpha f_3(\xi)$ is increasing. Then
$$
\mathfrak{L}(f_3)\ge (f_3^m)''+f_3^p-\alpha f_3\ge
(f_3^m)''+f_3^p(0)-\alpha f_3(0)= 0.
$$
\item For $\xi>\xi_1$ we consider $f_4=g_+$, where $g$ is  the solution to the initial value problem
$$
\left\{
\begin{array}{ll}
(g^m)''-\beta \xi g'+g^p-\alpha g=0,\qquad & \xi>\xi_1,\\
g(\xi_1)=f_3(\xi_1),\\
g'(\xi_1)=0.
\end{array}\right.
$$
It is clear that if $f_3(\xi_1)<(1/\alpha)^\alpha$ then $g$ is nonincresing and positive for $\xi_1\le\xi<\xi_2\le\infty$.
\end{enumerate}
The final function putting together $f_i$, $i=1,\cdots,4$ is a subsolution to our problem with zero initial value. This gives the desired lower bound of the grow-up rate.
\end{proof}

\section{Blow-up rates}\label{sect-buprates}

\begin{proof}[Proof of Theorem \ref{teo-buprates}] The lower blow-up rate is obtained easily by (strict) comparison with the supersolution
$$
U(t)=C_p (T-t)^{-\frac1{p-1}}.
$$
Indeed, if we assume that there exists $t_0\in(0,T)$ such that
$$
\|u(\cdot,t_0)\|_\infty < U(t_0),
$$
it also holds
$$
\|u(\cdot,t_0)\|_\infty < U(t_0-\varepsilon)
$$
for some $\varepsilon>0$, which is a contradiction with the fact that $u$ blows up at time $T$.

In order to prove the upper blow-up rate we use a rescaling technique inspired in the work \cite{FS}.

Let us define
\begin{equation}\label{M}
M(t)=\max_{\mathbb R\times [0,t]} u(x,\tau),
\end{equation}
and consider, for any fixed $t_0\in(0,T)$, the increasing sequence of times
$$
t_{j+1}=\sup\{t\in(t_j,T)\,:\, M(t)=2M(t_j)\}.
$$
Observe that for this sequence we have $\|u(\cdot,t_j)\|_\infty= M(t_j)$. We also observe that since the reaction only takes place for $x>0$, we get that near the blow-up time the maximum of $u$ is achieved in $\mathbb R^+$. Therefore we can take $x_j\ge 0$  such that
$$
u(x_j,t_j)= M(t_j).
$$
We consider the sequence
$$
z_j=(t_{j+1}-t_j){M^{p-1}(t_j)}.
$$
Let us observe that if $z_j$ is bounded, we get that
$$
t_{j+1}-t_j\le cM^{1-p}(t_j)=c2^{j(1-p)}M^{1-p}(t_0).
$$
Performing the sum,
$$
T-t_0\le c M^{1-p}(t_0)\sum_{j=0}^\infty2^{j(1-p)}=c'M^{1-p}(t_0)\le c'\|u(\cdot,t_0)\|_\infty^{1-p},
$$
that is, the desired upper blow-up rate. Therefore, in order to arrive at a contradiction, we assume that there exists a subsequence of times, still denoted  $t_j$, such that
\begin{equation}\label{eq.times}
\lim_{j\to\infty}z_j=\infty.
\end{equation}
Now we define the functions
$$
\varphi_j(y,s)=\frac1{M_j} u(M_j^{\frac{m-p}2 }y+ x_j,M_j^{1-p}s+ t_j),
$$
for
$$
y\in \mathbb R, \qquad s\in I_j= (- t_j M_j^{p-1}, (T- t_j)M_j^{p-1}),
$$
where $M_j=M(t_j)$. Notice that $I_j\to \mathbb R$ as  $j\to\infty$ and  $\varphi_j$ is a solution to the equation
$$
(\varphi_j)_s=(\varphi_j^m)_{xx}+a(y-x_j M_j^{\frac{p-m}2}) \varphi_j^p,\qquad  (y,s)\in \mathbb{R}\times I_j.
$$
It also satisfies
$$
\varphi_j(0,0)=1 \quad\text{and}\quad
\varphi_j(y,s)\le 2 \quad \mbox{in } \mathbb R\times I_j.
$$
The uniform bounds for $\varphi_j$ impliy that $\varphi_j$ is H\"older continuous with uniform coeffcient. Since $\varphi_j(0,0)=1$ we have a uniform nontrivial lower bound for every $\varphi_j$, that is,
\begin{equation}\label{lower}
\varphi_j(y,0)\ge g(y)\ge0,
\end{equation}
for some nontrivial function $g$.

We claim that, under the assumption of Theorem \ref{teo-buprates}, each function $\varphi_j$ blows up at a finite time $S_j$ which is uniformly bounded, that is $S_j<S$. This is a contradiction with the fact that $\varphi_j(y,s)\le 2$ for $s\in (0,z_j)$. Indeed, since $z_j\to\infty$ we can take $j$ large such that $z_j> S$. Therefore \eqref{eq.times} can not be true and the blow-up rate follows.

In order to prove the claim we first observe that since $x_j\ge0$, we have that  $\varphi_j$ is a supersolution of the equation
\begin{equation}\label{eq-h}
h_s=(h^m)_{yy}+ a(y) h^p, \qquad  (y,s)\in \mathbb R\times (0,z_j).
\end{equation}
Therefore, for $p\le m+2$, we can apply Theorem \ref{teo-exponents} to get that the solution of the above equation with initial datum $h(x,0)=g(y)$ blows up at some time $S$. Then, by comparison, $\varphi_j$ blows up at time $S_j<S$.

For the case $p>m+2$ we need the extra hypothesis $\partial_t u\ge0$, which implies $(\varphi_j)_s\ge0$. Therefore $\varphi_j$ is a supersolution of the problem
\begin{equation}\label{polub}
\left\{
\begin{array}{ll}
w_s=(w^m)_{yy}, \qquad & (y,s)\in \mathbb R^+\times (0,z_j),\\
w(0,s)=1, \\
w(y,0)=0.
\end{array}\right.
\end{equation}
Since $w\le 1$ we can pass to the limit, by means of a Lyapunov functional, to get that $w(x,s)\to 1$ as $s\to\infty$ uniformly in compact sets of $\mathbb R^+$. Actually in the linear case $m=1$ the solution to problem~\eqref{polub} is explicit, while if $m\neq1$ the solution is the so-called Polubarinova-Kochina solution.   Notice that this behaviour is also true if we consider the  problem in $\mathbb R^-$. Therefore, by comparison $\varphi_j(y,s)\ge 1/2$ in $|y|<K$ and $s>s_K$, and then
$$
\varphi_j(y,s_K)\ge h_0(y)=\frac12(1-x^2/K^2)_+^{1/m}.
$$
Observe that the energy of $h_0$ given in  \eqref{energy} satisfies
$$
E_{h_0}=\dfrac12\int_{-\infty}^\infty|(h_0^m)_x|^2-\frac m{p+m}\int_0^\infty h_0^{p+m}= C_1K^{-1}-
C_2 K <0
$$
for $K$ large. Then, applying the concavity argument the solution of \eqref{eq-h} with initial datum $h_0$
blows up at finite time $S$ and by comparison $\varphi_j$ also blows up at a time $S_j<S$.
\end{proof}


\section{Blow-up sets}\label{sect-bupsets}

We prove here Theorem \ref{teo-bupsets}. We first consider the case $p\ge m$.

\begin{lema}
Let $u$ be a blow-up solution to \eqref{1.1} with compactly supported initial datum. Then $B(u)$ is bounded from the left. In fact,

i) $B(u)\subset \mathbb R^+$ if $p>m$;

ii) $B(u)\subset [-K,\infty)$ if $p=m$.
\end{lema}
\begin{proof}
Notice that by the upper blow-up rate, $u$ is a subsolution to the problem on the left half-line
\begin{equation}\label{eq.pbw}
\left\{
\begin{array}{ll}
w_t=(w^m)_{xx}, \qquad & x<0,\,0<t<T,\\
w(0,t)=C (T-t)^{-1/(p-1)},\\
w(x,0)=w_0(x),
\end{array}\right.
\end{equation}
provided that $C$ is large. For $m=1$ we have and explicit formula for $w$, and it is easy to see that $w$ is bounded for $x<0$, see for instance \cite{SamarskiiBook87}. For $m\ne1$ we use comparison with a selfsimilar solution in the form
$$
W(x,t)=(T-t)^{-\alpha} F(x (T-t)^{-\beta}), \qquad \alpha=\frac1{p-1},\ \ \beta=\frac{p-m}2 \alpha,
$$
where the profile $F$ satisfies the equation
$$
(F^m)''-\beta \xi F-\alpha F=0.
$$
Observe that for $p=m$, i.e. $\beta=0$, this equation is the same as for problem \eqref{phi}, so by Lemma \ref{lem-phi} there exists a profile $F_1$ with compact support and satisfying $F_1(0)=1$. By scaling $F(\xi)=CF(C^{\frac{1-m}2}\xi)$ is also a solution, with large support if $C$ is large. We then take $C$ large so as to have that the corresponding solution $W$ satisfies $W(x,0)\ge u_0(x)$ and by comparison we obtain the bound of $B(u)$.

For the case $p>m$ ($\beta>0$), we introduce as in Seccion \ref{p=1<m} the variables
\begin{equation}\label{variable}
X=\frac{|\xi| g'}{g},\qquad Y=\frac1m \xi^2 g^{1-m},\qquad \eta=\log|\xi|,
\end{equation}
to obtain the differential system,
$$
\left\{
\begin{array}{l}
\dot{X}=X(1-mX)+Y(\alpha+\beta X),\\
\dot{Y}=Y(2-(m-1)X).
\end{array}\right.
$$
It is easy to see that all the orbits in the second quadrant start at the origin and have three posible behaviours: they cross the vertical axis; or the horizontal variable goes to $-\infty$; or  $(X,Y)\to (-\alpha/\beta,\infty)$. The existence of a unique orbit joining the origin with $(-\alpha/\beta,\infty)$ is given in \cite{Rocky} for $m<1$, but the argument works as well for $m>1$. From this orbit we obtain a positive, increasing profile $F_1$  such that $F_1(0)=1$ and $F_1(\xi)\sim |\xi|^{-\alpha/\beta}$ for $\xi\sim -\infty$. Notice that, for $x<0$ and $t$ near $T$, we have
$$
W_1(x,t)=(T-t)^{-\alpha} F_1(x (T-t)^{-\beta})\sim |x|^{-\alpha/\beta},
$$
that is $W$ is bounded. Rescaling and comparison as before implies the same property for our solution, $u$ is bounded for $x<0$.
\end{proof}

\begin{rem}\label{rem-p<m}
Notice that for $p<m$ we  have again  an equation like \eqref{phi} for the profile.  This give us a family of selfsimilar solutions
$$
W_C(x,t)=C (T-t)^{-\frac1{p-1}} F_1(C^{\frac{1-m}{2}}x (T-t)^{\frac{m-p}{2(p-1)}})
$$
with global blow-up.
\end{rem}

\begin{lema}
Let $u$ be a blow-up solution to \eqref{1.1} with compactly supported initial datum. Then $B(u)$ is bounded from the right provided $p\ge m>1$.
\end{lema}
\begin{proof}
We only have to notice that  the support of the blow-up solutions to the equation with global reaction,
$a(x)=1$,
are bounded if $p\ge m>1$, see \cite{SamarskiiBook87}. The proof of this result uses the intersection comparison technique with self-similar profiles in a neighborhood of the free boundary. Then the same result holds for our equation near the right-hand free boundary.
\end{proof}

\begin{lema}
Let $u$ be a blow-up solution to \eqref{1.1} with $p<m$ and a compactly supported initial datum. Assume also that there exists $x_0\in\mathbb R$ satisfying \eqref{x0}. Then $B(u)=\mathbb R$.
\end{lema}
\begin{proof}
Thanks to the hypothesis at $x_0$ we have that $u$ is a supersolution to the problem defined on the left of $x_0$,
$$
\left\{
\begin{array}{ll}
w_t=(w^m)_{xx}, \qquad & x<x_0,\,0<t<T,\\
w(x_0,t)=C_1 (T-t)^{-1/(p-1)},\\
w(x,0)=w_0(x),
\end{array}\right.
$$
as well as  to the problem on the right, $x>x_0$. Moreover, the self-similar solution given in Remark \ref{rem-p<m} is a subsolution if we choose $C$ small enough such that $W_C(x,0)<u_0(x)$. Comparison ends the proof.
\end{proof}
The same argument allows to prove that if $p=m$ the blow-up set contains some nontrivial interval, thus concluding the proof of Theorem~\ref{teo-bupsets}.

To finish this section we remark that, in the range $p<m$, but without the hypothesis of the existence of $x_0$, it is easy to see that $B(u)$ is unbounded at least from the right. In fact if we assume that at some point $u(x_1,t)<M$, then $u$ is a subsolution to
$$
\left\{
\begin{array}{ll}
z_t=(z^m)_{xx}+a(x) z^p, & x<x_1,\, t>0,\\
z(x_1,t)=M, \\
z(x,0)=z_0(x).
\end{array}\right.
$$
On the other hand, the stationary solution of the equation with $z(0)=K$ and $z'(0)=0$, is  supersolution of the problem if $K$ is large enough. Then by comparison $u$ must be bounded in $(-\infty,x_1)$. Therefore if $u$ is bounded in some interval $(x_1,\infty)$ then it cannot blow up.




\begin{thebibliography}{99}                                                                                               %

\bibitem{AE86}
{\sc J. Aguirre and M. Escobedo.}
{\it A Cauchy problem for $u_t-\Delta u=u^p$ with $0<p<1$. Asymptotic behaviour
of solutions.}  Ann. Fac. Sci. Toulouse Math. {\bf8}  (1986), 175--203.

\bibitem{BaiZhouZheng}
{\sc X. Bai, S. Zhou and  S. Zheng.}
{\it Cauchy problem for fast diffusion equation with localized reaction.} Nonlinear Anal.
{\bf74} (2011), 2508--2514.

\bibitem{FdP}
{\sc R. Ferreira and A. de Pablo}. {\it Grow-up for a quasilinear heat equation with a localized reaction.} J. Differential Equations {\bf268} (2020), 6211--6229.


\bibitem{FdPV06} {\sc R. Ferreira, A. de Pablo and J.L. V\'{a}zquez}. {\it Blow-up
for the porous medium equation with a localized reaction}. J. Differential Equations  {\bf231} (2006),  195--211

\bibitem{Rocky} {\sc R. Ferreira, A. de Pablo, F. Quirós and J. D. Rossi}.{\it The blow-up profile for a fast diffusion equation with a nonlinear boundary condition}. Rocky Mountain J. Math. {\bf 33} (2003), 123–146

\bibitem{FS} {\sc M. Fila and P. Souplet}.{\it The blow-up rate for semilinear parbolic problems on general domains}. Nonlinear differ. equ. appl. {\bf 8} (2001), 473--480.

\bibitem{Galaktionov94} {\sc V.A. Galaktionov}. {\it Blow-up for quasilinear heat equations with critical Fujita's exponents}. Proc. Roy. Soc. Edinburgh
Sect. A 124 (1994) 517--525.

\bibitem{GP} {\sc B. H. Gilding and L. A. Peletier}. {\it On a class of similarity solutions of the porous media equation}. J. Math. Anal. Appl. {\bf 55} (1976), 351--364

\bibitem{HP} {\sc M.A. Herrero and M. Pierre}. {\it The Cauchy Problem for $u_t=\Delta u^m$ when
$0 < m < 1$}. Trans. Amer. Math. Soc.  {\bf 291} (1985), 145--158.
	
	
\bibitem{LevineSacks84}
{\sc H.A. Levine and P. Sacks}. Some existence and nonexistence theorems for solutions of degenerate parabolic equations.
J. Differential Equations 52 (1984) 135--161.
%

\bibitem{dPV90}
{\sc A. de Pablo and J.~L. V\'azquez}. {\it The balance between strong reaction and slow diffusion}.  Comm. Partial Differential Equations {\bf15} (1990),  159--183.

\bibitem{dPV91}
{\sc A. de Pablo and J.~L. V\'azquez}. {\it Travelling waves and finite propagation in a reaction-diffusion equation}. J. Differential Equations 93 (1991), no. 1, 19--61.

\bibitem{Pinsky}
{\sc R.G. Pinsky}. {\it Existence and nonexistence of global solutions for $u_t = \Delta u + a(x)u^p$ in $R^d$}.  J. Differential Equations {\bf133} (1997), 152--177

\bibitem{SamarskiiBook87} {\sc A.A. Samarskii, V.A. Galaktionov, S.P. Kurdyumov and A.P. Mikhailov}. Blow-up in Problems for Quasilinear Parabolic
Equations. Nauka, Moscow, 1987 (in Russian); English translation: Walter de Gruyter, Berlin, 1995.


\bibitem{vazquez} {\sc J.~L. V\'azquez}.
\lq\lq The porous medium equation. Mathematical theory''.    Oxford
Mathematical Monographs. The Clarendon Press, Oxford University
Press, Oxford, 2007. ISBN: 978-0-19-856903-9.
\end{thebibliography}
\end{document}